\documentclass[12pt]{amsart}
\usepackage[margin=1.54in]{geometry}

\usepackage{amssymb, amsmath}
\usepackage{amsthm}
\usepackage{enumerate}
\usepackage{cases}
\usepackage{verbatim}

\newtheorem{theorem}{Theorem}[section]
\newtheorem{lemma}[theorem]{Lemma} 
\newtheorem{proposition}[theorem]{Proposition} 
\newtheorem{thmletter}{Theorem}

\newcommand{\p}[1]{\noindent {\newline\bf #1.}}
\newcommand{\out}{\operatorname{Out}}
\newcommand{\aut}{\operatorname{Aut}}
\newcommand{\inn}{\operatorname{Inn}}
\newcommand{\<}{\langle}
\renewcommand{\>}{\rangle}
\newcommand{\paper}{paper}

\begin{document}


\title[Outer automorphism groups of residually finite groups]
{On the outer automorphism groups of finitely generated, residually finite groups}

\author
{Alan D.\ Logan}

\address{School of Mathematics and Statistics\\
University Gardens \\ University of Glasgow \\ G12 8QW \\
Scotland.}
\email{alan.logan@glasgow.ac.uk}

\date{\today}

\keywords{Outer automorphism groups, Residual finiteness, Mapping tori}

\subjclass[2000]{20F28, 20E99}

\maketitle

\begin{abstract}
Bumagin--Wise posed the question of whether every countable group can be realised as the outer automorphism group of a finitely generated, residually finite group. We give a partial answer to this problem for recursively presentable groups.
\end{abstract}



\section{Introduction}
\label{introduction}
Every group can be realised as the outer automorphism group of some group \cite{matumoto1989any}. One can ask what restrictions can be placed on the groups involved. Notably, Bumagin--Wise proved that every countable group $Q$ can be realised as the outer automorphism group of a finitely generated group $G_Q$ \cite{BumaginWise2005}. Several other authors have achieved results in a similar vein (see, for example, \cite{kojima1988isometry}, \cite{gobel2000outer}, \cite{droste2001all}, \cite{braun2003outer}, \cite{frigerio2005countable}, \cite{minasyan2009groups}).

To prove their result, Bumagin--Wise construct $G_Q$ as the kernel of a short exact sequence using a version of a  Rips' construction \cite{rips1982subgroups}.
Their proof also shows that if $Q$ is finitely presented then $G_Q$ can be taken to be residually finite. They then pose the question: can every countable group $Q$ be realised as the outer automorphism group of a finitely generated, residually finite group $G_Q$?

In this paper we give a partial answer to this question of Bumagin--Wise. Our proof is based upon the construction of Bumagin--Wise and utilises an embedding of Sapir \cite{sapir2014higman}.

\begin{thmletter}
\label{thm:intro1}
If $Q$ is a finitely generated, recursively presented group then either $Q$ or $Q\times C_2$ can be realised as the outer automorphism group of a finitely generated, residually finite group $G_Q$.
\end{thmletter}

This theorem admits a possible improvement: Osin asked if every finitely generated, recursively presentable group can be embedded as a malnormal subgroup of a finitely presented group, and a positive answer to this question would allow us to dispense with the $Q\times C_2$ possibility. Sapir has recently stated that his embedding yields such a positive solution.

\begin{thmletter}
\label{thm:intro2}
Suppose every finitely generated, recursively presented group $Q$ can be embedded as a malnormal subgroup of a finitely presented group $H_Q$. Then every finitely generated, recursively presented group $Q$ can be realised as the outer automorphism group of a finitely generated, residually finite group $G_Q$.
\end{thmletter}

\p{Outline of the \paper} In Section~\ref{sec:OutOfMT} we prove a technical theorem, Theorem~\ref{theorem:secondmainthm}, which classifies the outer automorphism group of mapping tori $H_{\phi}=H\rtimes_{\phi}\mathbb{Z}$ where $H$ has trivial centre and has no epimorphisms onto $\mathbb{Z}$. In Section~\ref{sec:bumaginwise} we use this technical result to obtain a way of ``grabbing'' a finitely generated subgroup of $\out(H)$, which is applied to prove our main theorems, Theorems~\ref{thm:intro1}~and~\ref{thm:intro2}. 

\p{Acknowledgements} The author would like to thank his PhD supervisor, Stephen J. Pride, and Tara Brendle for many helpful discussions about this \paper. He would also like to thank Mark Sapir and Ian Agol for pointing out on MathOverflow the embedding used in Section~\ref{sec:bumaginwise}, Henry Wilton for informing the author of the existence of finitely generated, residually finite groups which are not recursively presentable (such groups are used in the proof of Proposition~\ref{prop:fgrf}), and Ashot Minasyan for his helpful discussions on recursive presentability and for allowing the author to include the proof of Proposition~\ref{prop:Minasyan}. The author was supported by an EPSRC doctoral training grant.


\section{The outer automorphism groups of mapping tori}
\label{sec:OutOfMT}
For $\phi\in\aut(H)$ an automorphism of $H$ we shall write $H_{\phi}=H\rtimes_{\phi}\mathbb{Z}$ for the mapping torus $\<H, t; tht^{-1}=\phi(h), \:h\in H\>$ associated to $\phi$. In this section we prove Theorem~\ref{theorem:secondmainthm}, which is the main technical result of this paper. For $H$ a group with no epimorphisms onto $\mathbb{Z}$ and with trivial center, this theorem gives a description of the outer automorphism group $\out(H_{\phi})$ of a mapping torus $H_{\phi}$. Theorem~\ref{theorem:secondmainthm} forms the basis of the proof of Theorems~\ref{thm:intro1}~and~\ref{thm:intro2}, which are the main theorems of this paper.

The layout of this current section is as follows. We begin by proving, in Lemma~\ref{lem:PettetnonontoZ}, that, because $H$ has no epimorphisms onto $\mathbb{Z}$, every automorphism of the mapping torus $H_{\phi}$ fixes the subgroup $H$. We then use this to prove, in Lemma~\ref{lem:principle}, that the representatives for elements of $\out(H_{\phi})$ can be taken to have a specific form. In Section~\ref{sec:mappingtori} we use the representatives given by Lemma~\ref{lem:principle} to prove Theorem~\ref{theorem:secondmainthm}.


\subsection{Automorphisms fix the base group}
Consider a mapping torus $H_{\phi}=H\rtimes_{\phi}\mathbb{Z}$ such that $H$ has no epimorphisms onto $\mathbb{Z}$. The following lemma appears in a paper of Arzhantseva--Lafont--Minasyan \cite{arzhantseva2014isomorphism}, although it is somewhat hidden in the proof of their Proposition~$2.1$.

\begin{lemma}[Arzhantseva--Lafont--Minasyan]
\label{lem:PettetnonontoZ}
Suppose $H$ has no epimorphisms onto $\mathbb{Z}$. Then every automorphism of a mapping torus $H_{\phi}=H\rtimes_{\phi}\mathbb{Z}$ maps $H$ to itself.
\end{lemma}

\begin{proof}
Consider the following composition of maps, where the first embedding is the natural one of $H$ into $H_{\phi}$, where the map $\psi:H_{\phi}\rightarrow H_{\phi}$ is an automorphism of $H_{\phi}$, and where the final surjection is the natural one of $H_{\phi}$ onto $\mathbb{Z}$ by quotienting out $H$.
\[
H\hookrightarrow H_{\phi} \xrightarrow{\psi} H_{\phi}\twoheadrightarrow \mathbb{Z}
\]
As $H$ does not map onto $\mathbb{Z}$, these maps compose to give the trivial map. Therefore, $H\psi\leq H$. Using the same argument with $\psi^{-1}$, we see that $H\psi^{-1}\leq H$ and so $H\psi=H$ as required.
\end{proof}

Note that for $H$ an arbitrary group, the automorphisms of a mapping torus $H_{\phi}=H\rtimes_{\phi}\mathbb{Z}$ which fix $H$ form a subgroup $\aut_H(H_{\phi})$ of $\aut(H_{\phi})$, and this subgroup contains all the inner automorphisms so there is an analogous subgroup $\out_H(H_{\phi})$ of $\out(H_{\phi})$. The work in the remainder of Section~\ref{sec:OutOfMT} can be viewed as studying this subgroup $\out_H(H_{\phi})$ of $\out(H_{\phi})$. Lemma~\ref{lem:PettetnonontoZ} proves that $\out_H(H_{\phi})=\out(H_{\phi})$ in our particular case.


\subsection{The form of (outer) automorphisms}\label{subsection31}
Consider a mapping torus $H_{\phi}=H\rtimes_{\phi}\mathbb{Z}=\<H, t; tht^{-1}=\phi(h), \:h\in H\>$ such that $H$ has no epimorphisms onto $\mathbb{Z}$ and $H$ has trivial center. Our main technical theorem, Theorem~\ref{theorem:secondmainthm}, follows from a classification of the elements of $\out(H_{\phi})$, that is, to prove the main technical theorem we begin by finding representative automorphisms for elements of $\out(H_{\phi})$. The purpose of this current section, Section~\ref{subsection31}, is to prove Lemma~\ref{lem:principle}, which gives this classification.

We begin by proving that certain maps, which are used as representatives for elements of $\out(H_{\phi})$ in Lemma~\ref{lem:principle}, define automorphisms of $H_{\phi}$. There are two forms these representatives take, and Lemma~\ref{lem:alphaform} considers the first form while Lemma~\ref{lem:zetaform} considers the second form.

We shall write $[h,k]=h^{-1}k^{-1}hk$, and by $\delta\phi(h)$ we mean $\phi(\delta(h))$. We shall write multiplication in $G$ as $g h$ and as $g\cdot h$, with the latter notation being used to ease any ambiguity occurring when considering the images of elements under automorphisms, for example we would write $g\cdot\phi\psi(h)\cdot k$. For $K$ some group with automorphism $\psi\in\aut(K)$, we shall write $\widehat{\psi}$ for the element of $\out(K)$ with representative $\psi$. We shall write $C_K(g)$ to denote the centraliser of the element $g\in K$.

\begin{lemma}\label{lem:alphaform}
Let $H_{\phi}$ be a mapping torus of $H$. If ${\delta}\in \operatorname{Aut}(H)$ is such that $\widehat{\delta}\in C_{\out(H)}(\widehat{\phi})$ then $\delta$ induces an automorphism of $H_{\phi}$ in the following way, where $g$ is such that $[{\delta},\phi]=\gamma_g$.
\begin{align*}\alpha_{\delta}:
h&\mapsto {\delta}(h)&\forall h\in H\\
t&\mapsto gt \end{align*}
\end{lemma}

\begin{proof}
%
To see that $\alpha_{\delta}$ is a homomorphism note that it satisfies all the relators of $H$, as $\alpha_{\delta}|_H\in\operatorname{Aut}(H)$, so it is sufficient to prove that $\alpha_{\delta}(th)=\alpha_{\delta}(\phi(h)\cdot t)$ for all $h\in H$. So, the left hand side is as follows.
\begin{align*}
\alpha_{\delta}(th)
&=gt\cdot \delta(h)\\
&=g\cdot\delta\phi(h)\cdot t&\forall h\in H
\end{align*}
We now evaluate the right hand side as follows. Note that (\ref{eqn:checkAlphaHom2}), below, is obtained because $g$ is such that $\delta\phi\gamma_g^{-1}=\phi\delta$.
\begin{align}
\alpha_{\delta}(\phi(h)\cdot t)
&=\phi\delta(h)\cdot gt\notag\\
&=\delta\phi\gamma_g^{-1}(h)\cdot gt\label{eqn:checkAlphaHom2}\\
&=g\cdot\delta\phi(h)\cdot t&\forall h\in H\notag
\end{align}
The left and right hand sides are equal, so $\alpha_{\delta}$ is a homomorphism.

To see that $\alpha_{\delta}$ is surjective, note that its restriction to $H$ is surjective, and further note that $t\mapsto gt$ for some $g\in H$ so $t$ is in the image.

To see that $\alpha_{\delta}$ is right-invertible, and so injective, we note that $\alpha_{\delta^{-1}}$ is also a homomorphism and then prove that $\alpha_{\delta}\alpha_{\delta^{-1}}$ is trivial. So, $\alpha_{\delta^{-1}}$ is a homomorphism as $\widehat{\delta^{-1}}\in C_{\out(H)}(\widehat{\phi})$ because $\widehat{\delta}\in C_{\out(H)}(\widehat{\phi})$. Now, because $[\delta, \phi]=\gamma_g$ we have that $[\delta^{-1}, \phi]=\gamma_{\delta^{-1}(g^{-1})}$, which means that $\alpha_{\delta^{-1}}(t)=\delta^{-1}(g^{-1})\cdot t$. Then, $\alpha_{\delta^{-1}}$ is the right inverse of $\alpha_{\delta}$ as clearly $\alpha_{\delta}\alpha_{\delta^{-1}}$ fixes $h$ for all $h\in H$ while we have the following.
\begin{align*}
\alpha_{\delta}\alpha_{\delta^{-1}}(t)
&=\alpha_{\delta^{-1}}(gt)\\
&=\delta^{-1}(g)\cdot \delta^{-1}(g^{-1})\cdot t\\
&=t
\end{align*}
Therefore, $\alpha_{\delta}$ is injective. The proof of the lemma is complete.
\end{proof}

The second form which automorphisms can take is given by Lemma~\ref{lem:zetaform}.

\begin{lemma}\label{lem:zetaform}
Let $H_{\phi}$ be a mapping torus of $H$. If ${\delta}\in \operatorname{Aut}(H)$ is such that $\widehat{\delta}^{-1}\widehat{\phi}\widehat{\delta}=\widehat{\phi}^{-1}$ then $\delta$ induces an automorphism of $H_{\phi}$ in the following way, where $g$ is such that $\delta^{-1}\phi\delta=\phi^{-1}\gamma_g$.
\begin{align*}\zeta_{\delta}:
h&\mapsto \delta(h)&\forall h\in H\\
t&\mapsto g^{-1}t^{-1} \end{align*}
\end{lemma}

\begin{proof}
%
To see that $\zeta_{\delta}$ is a homomorphism note that it satisfies all the relators of $H$, as $\zeta_{\delta}|_H\in\operatorname{Aut}(H)$, so it is sufficient to prove that $\zeta_{\delta}(th)=\zeta_{\delta}(\phi(h)\cdot t)$ for all $h\in H$. So, the left hand side is as follows.
\begin{align*}
\zeta_{\delta}(th)
&=g^{-1}t^{-1}\cdot\delta(h)\\
&=g^{-1}\cdot\delta\phi^{-1}(h)\cdot t^{-1}
\end{align*}
We now evaluate the right hand side as follows. Note that (\ref{eqn:checkZetaHom2}), below, is obtained because $g$ is such that $\delta\phi^{-1}\gamma_g=\phi\delta$.
\begin{align}
\zeta_{\delta}(\phi(h) t)
&=\phi\delta(h) \cdot g^{-1}t^{-1}\notag\\
&=\delta\phi^{-1}\gamma_g(h)\cdot g^{-1}t^{-1}\label{eqn:checkZetaHom2}\\
&=g^{-1}\cdot\delta\phi^{-1}(h)t^{-1}\notag
\end{align}
The left and right hand sides are equal, so $\zeta_{\delta}$ is a homomorphism.

To see that $\zeta_{\delta}$ is surjective, note that its restriction to $H$ is surjective, and further note that $t\mapsto g^{-1}t^{-1}$ for some $g\in H$ so $t$ is in the image.

In order to prove that $\zeta_{\delta}$ is right-invertible, and so injective, we shall prove that $\alpha_{(\delta^2\gamma_{g}^{-1})}$ is an automorphism of $H_{\phi}$ and that $\zeta_{\delta}^2\gamma_{g}^{-1}=\alpha_{(\delta^2\gamma_g^{-1})}$.
We begin by evaluating $[\delta^2\gamma_g^{-1}, \phi]$ as follows, where (\ref{autequalities1}) is obtained because $\delta^{-1}\phi^{-1}\delta=\gamma_g^{-1}\phi$, while $\phi\delta\phi\gamma_{\phi(g)}^{-1}=\delta$ yields (\ref{autequalities2}).
\begin{align}
[\delta^2\gamma_g^{-1}, \phi]
&=\gamma_g\delta^{-2}\phi^{-1}\delta^2\gamma_g^{-1}\phi\notag\\
&=\gamma_g\delta^{-1}(\delta^{-1}\phi^{-1}\delta)\delta\phi\gamma_{\phi(g)}^{-1}\notag\\
&=\gamma_g\delta^{-1}\gamma_g^{-1}(\phi\delta\phi\gamma_{\phi(g)}^{-1})\label{autequalities1}\\
&=\gamma_g\delta^{-1}\gamma_g^{-1}\delta\label{autequalities2}\\
&=\gamma_g\gamma_{\delta(g^{-1})}=\gamma_{g\cdot\delta(g^{-1})}\notag
\end{align}
This implies that $\widehat{\delta^2\gamma_g^{-1}}\in C_{\out(H)}(\widehat{\phi})$, so by Lemma~\ref{lem:alphaform} we have that $\alpha_{(\delta^2\gamma_g^{-1})}\in\aut(H_{\phi})$. Note that it also implies the following.
\[
\alpha_{(\delta^2\gamma_g^{-1})}(t)=g\cdot \delta(g^{-1})\cdot t
\]
Then, to prove that $\zeta_{\delta}^2\gamma_{g}^{-1}=\alpha_{(\delta^2\gamma_g^{-1})}$, note that as their restriction to $H$ is identical and because $\zeta_{\delta}^2$ is a homomorphism, it is sufficient to prove that $\zeta_{\delta}^2\gamma_{g}^{-1}(t)=g\cdot \delta(g^{-1})\cdot t$. We have the following.
\begin{align*}
\zeta_{\delta}^2\gamma_{g}^{-1}(t)
&=\zeta_{\delta}\gamma_{g}^{-1}({g}^{-1}t^{-1})\\
&=\displaystyle\gamma_{g}^{-1}\left(\delta(g^{-1})\cdot tg\right)\\
&=\displaystyle g\cdot \delta(g^{-1})\cdot t
\end{align*}
We conclude that $\zeta_{\delta}^2\gamma_g^{-1}=\alpha_{(\delta^2\gamma_g^{-1})}$, so the lemma holds.
\end{proof}

\p{Classifying the elements of \boldmath{$\out(H_{\phi})$}}
We shall now prove Lemma \ref{lem:principle}, which classifies the coset representatives for $\out(H_{\phi})$. Proving this lemma is the purpose of this current section, Section~\ref{subsection31}.

\begin{lemma}
\label{lem:principle}
Suppose $H_{\phi}=H\rtimes_{\phi}\mathbb{Z}$ is a mapping torus such that $H$ has no epimorphisms onto $\mathbb{Z}$.
Then every element $\widehat{\psi}$ of $\out(H_{\phi})$ has a representative in $\aut(H_{\phi})$ of the form $\alpha_{\delta}$ or of the form $\zeta_{\delta}$. Moreover, every map $\alpha_{\delta}$ and $\zeta_{\delta}$ defines an automorphism of $H_{\phi}$.
\end{lemma}

\begin{proof}
By Lemma~\ref{lem:alphaform} and Lemma~\ref{lem:zetaform}, each of the prospective representatives $\alpha_{\delta}$ and $\zeta_{\delta}$ define automorphisms of $H_{\phi}$. Therefore, we prove, below, the first part of this lemma, that every element of $\out(H_{\phi})$ has a representative in $\aut(H_{\phi})$ of one of the stipulated forms.

We begin by proving that if $\widehat{\psi}\in\operatorname{Out}(H_{\phi})$ then there is a representative $\psi\in\aut(H_{\phi})$ of the following form, where $g\in H$ and $\delta\in\aut(H)$.
\begin{align}
\psi: h&\mapsto \delta(h)&h\in H \label{align:roughformofauts}\\
t&\mapsto gt^{\epsilon}\notag
\end{align}
To see this, consider a representative $\psi\in\widehat{\psi}$. Note that $\psi(H)=H$ by Lemma~\ref{lem:PettetnonontoZ}, thus the restriction of $\psi$ to $H$ is an automorphism $\delta$ of $H$. Therefore, as $H_{\phi}=H \rtimes_{\phi}\mathbb{Z}$ is a semidirect product, the representative $\psi\in\aut(H_{\phi})$ can be chosen to be such that $\psi(h)=\delta(h)$ for all $h\in H$, and $\psi(t)=gt^{i}$ where $g\in H$ and $\delta\in\aut(H)$. We shall now prove that the number $i$ has absolute value one, $|i|=1$. This completes our proof that a representative $\psi\in\widehat{\psi}$ can be chosen to have the form (\ref{align:roughformofauts}). To see that $|i|=1$, note that, because $\psi$ is an automorphism, there exists a word $W$ over $H$ and $gt^i$ which represents $t$, $W(gt^i, H)=t$. However, as $H_{\phi}$ is a semidirect product this word can be written as $t^{ij}k$ for some $k\in H$, $j\in\mathbb{Z}$. Thus, $t=t^{ij}k$, and so $|i|=1$ as required.

We shall use the form (\ref{align:roughformofauts}) to prove the lemma. We investigate the cases $\epsilon=1$ and $\epsilon=-1$ separately.

Suppose $\epsilon=1$.
It is sufficient to prove that $\delta\phi=\phi\delta\gamma_g$ holds.
We have the following.
\begin{align*}
\psi(th)&=\psi\displaystyle\left(\phi(h)\cdot t\displaystyle\right)&\forall h\in H\\
gt\cdot \delta(h)&=\phi\delta(h)\cdot gt&\forall h\in H\\
g\cdot \delta\phi(h)\cdot t&=\phi\delta(h)\cdot gt&\forall h\in H
\end{align*}
Then, $\delta\phi=\phi\delta\gamma_g$ holds, so $\psi=\alpha_{\delta}$.

Suppose $\epsilon=-1$.
It is sufficient to prove that $\delta^{-1}\phi\delta=\phi^{-1}\gamma_g^{-1}$ holds (note that $g$ has been replaced with $g^{-1}$ in the definition of $\zeta_{\delta}$, as $\psi(t)=gt^{-1}$).
We have the following.
\begin{align*}
\psi(th)&=\psi\displaystyle\left( \phi(h)\cdot t\displaystyle\right)&\forall h\in H\\
gt^{-1}\cdot\delta(h)&=\phi\delta(h)\cdot gt^{-1}&\forall h\in H\\
g\cdot\delta\phi^{-1}(h)\cdot t^{-1}&=\phi\delta(h)\cdot gt^{-1}&\forall h\in H
\end{align*}
Then, $\delta\phi^{-1}\gamma_g^{-1}=\phi\delta$ holds, which yields the required equality, so $\psi=\zeta_{\delta}$. This completes the proof of the lemma.
\end{proof}


\subsection{The subgroup \boldmath{$\out^0(H_{\phi})$}}
\label{sec:SubgroupWeIsolate}
Having proven Lemma \ref{lem:principle}, we know, in a certain sense, what the elements of $\operatorname{Out}(H_{\phi})$ are, where $H_{\phi}=H\rtimes_{\phi}\mathbb{Z}$ is a mapping torus and $H$ does not map onto $\mathbb{Z}$. In Section~\ref{sec:mappingtori}, below, we analyse the group $\out^0(H_{\phi})$ consisting of the elements of the form $\widehat{\alpha}_{\delta}$, where ${\alpha}_{\delta}$ was defined in Lemma~\ref{lem:alphaform}, under the additional assumption that $H$ has trivial center, and this analysis yields Theorem~\ref{theorem:secondmainthm}. Note that the purpose of Section~\ref{sec:OutOfMT} is to prove Theorem~\ref{theorem:secondmainthm}, and this result forms the basis of the proofs of the main theorems, Theorems~\ref{thm:intro1}~and~\ref{thm:intro2}.

We shall now explain why we do not consider the automorphisms $\zeta_{\delta}$, but instead restrict our investigations to the subgroup $\out^0(H_{\phi})$ of $\out(H_{\phi})$. If there does not exist any automorphisms of the form $\zeta_{\delta}$ then $\operatorname{Out}^0(H_{\phi})=\operatorname{Out}(H_{\phi})$. Otherwise, noting that the $\alpha_{\delta}$ maps $t$ to $g_1t$ while $\zeta_{\delta^{\prime}}$ maps $t$ to $g_2^{-1}t^{-1}$ for some $g_1, g_2\in H$, we see that $\out^0(H_{\phi})$ is an index two subgroup of $\out(H_{\phi})$. Therefore, applying Lemma~\ref{lem:zetaform}, which provides conditions for the existence of a map $\zeta_{\delta}$, we have the following lemma.

\begin{lemma}
\label{lem:index1or2}
The subgroup $\out^0(H_{\phi})$ consisting of the outer automorphisms of the form $\widehat{\alpha}_{\delta}$ has index two in $\out(H_{\phi})$ if $\widehat{\phi}$ is conjugate to $\widehat{\phi}^{-1}$ in $\out(H)$. Otherwise, $\out^0(H_{\phi})=\out(H_{\phi})$.
\end{lemma}

This lemma is why in Theorem~\ref{theorem:secondmainthm} we restrict our analysis to $\operatorname{Out}^0(H_{\phi})$.
Note that the automorphisms of the form $\alpha_{\delta}$ are such that the following hold. We use these equalities throughout the remainder of Section~\ref{sec:OutOfMT}.
\begin{align*}
\alpha_{\delta}\alpha_{\xi}&=\alpha_{\delta\xi}\\
\alpha_{\delta}^{-1}&=\alpha_{\delta^{-1}}
\end{align*}


\subsection{The outer automorphism groups of certain mapping tori}
\label{sec:mappingtori}
Take $H_{\phi}=H\rtimes_{\phi}\mathbb{Z}$ to be a mapping torus with base group $H$ and associated automorphism $\phi\in\aut(H)$, and also assume that $H$ has trivial center and has no epimorphisms onto $\mathbb{Z}$. In this section we prove Theorem~\ref{theorem:secondmainthm}, which gives a description of $\out(H_{\phi})$ for such a group $H_{\phi}$. Recall that for $K$ a group and $\psi\in\aut(K)$, $\widehat{\psi}$ denotes the element of $\out(K)$ with representative $\psi$, and that $C_K(g)$ denotes the centraliser of the element $g\in K$.

\begin{theorem}\label{theorem:secondmainthm}
Let $H_{\phi}=H\rtimes_{\phi}\mathbb{Z}$ be a mapping torus with base group $H$ and associated automorphism $\phi$.
Assume $H$ has trivial center and has no epimorphisms onto $\mathbb{Z}$.
Then we have the following isomorphism,
\[
\operatorname{Out}^0(H_{\phi})\cong \frac{C_{\out(H)}(\widehat{\phi})}{\<\widehat{\phi}\>}
\]
where either $\operatorname{Out}^0(H_{\phi})=\operatorname{Out}(H_{\phi})$ or $\widehat{\phi}$ is conjugate to $\widehat{\phi}^{-1}$ in $\out(H)$, whence $\operatorname{Out}^0(H_{\phi})$ has index two in $\operatorname{Out}(H_{\phi})$.
\end{theorem}

\begin{proof}
By Lemma~\ref{lem:index1or2}, $\out^0(H_{\phi})$ has index one or two in $\out(G)$, and further has index two precisely when $\widehat{\phi}$ is conjugate to $\widehat{\phi}^{-1}$ in $\out(H)$, as required. We shall now prove that $\out^0(H_{\phi})$ is isomorphic to $C_{\out(H)}(\widehat{\phi})/\<\widehat{\phi}\>$, which completes the proof of the theorem.

Consider the following map. We shall prove that it is a well-defined surjective homomorphism with kernel $\<\widehat{\phi}\>$, which proves the theorem.
\begin{align*}
\eta: C_{\out(H)}(\widehat{\phi}) &\rightarrow \operatorname{Out}^0(H_{\phi})\\
\widehat{\delta} &\mapsto \widehat{\alpha}_{\delta}
\end{align*}
Note that the map $\eta$ is surjective by the definition of $\out^0(H_{\phi})$, and it is a homomorphism because $\alpha_{\delta}\alpha_{\xi}=\alpha_{\delta\xi}$.

To see that $\eta$ is well-defined, suppose that $\delta_2=\delta_1\gamma_k$. Note that $[\delta_1, \phi]=\gamma_{kg_2\cdot\phi(k^{-1})}$ where $g_2$ is such that $[\delta_2, \phi]=\gamma_{g_2}$. Then $\alpha_{\delta_2}(h)=\alpha_{\delta_1}\gamma_k(h)$ for all $h\in H$, while $\alpha_{\delta_2}(t)=g_2t$ and we have the following.
\begin{align*}
\alpha_{\delta_1}\gamma_k(t)
&=k^{-1}kg_2\cdot\phi(k^{-1})\cdot tk\\
&=g_2t
\end{align*}
We thus have that $\alpha_{\delta_2}=\alpha_{\delta_1}\gamma_k$, so $\widehat{\alpha}_{\delta_2}=\widehat{\alpha}_{\delta_1}$ as required.

Finally, to prove that the map $\eta$ has kernel $\<\widehat{\phi}\>$ begin by supposing that $\alpha_{\delta}$ is inner, and so $\alpha_{\delta}=\gamma_{kt^i}$ for some $k\in H$ and $i\in\mathbb{Z}$. This means that $h=t^{-i}k^{-1}\cdot\delta^{-1}(h)\cdot kt^i$ for all $h\in H$, so $\phi^i(h)=\delta^{-1}\gamma_k(h)$ for all $h\in H$, and so $\widehat{\delta}=\widehat{\phi}^j$ in $\out(H)$ for some $j\in\mathbb{Z}$. Therefore, $\ker\eta\leq\<\widehat{\phi}\>$. On the other hand, $\alpha_{\phi}$ is inner because $\alpha_{\phi}(h)=\phi(h)=tht^{-1}$ while $\alpha_{\phi}(t)=t$. Therefore, $\<\widehat{\phi}\>\leq\ker\eta$. Thus, we conclude that $\alpha_{\delta}\in\inn(H_{\phi})$ if and only if $\widehat{\delta}\in\<\widehat{\phi}\>$.
\end{proof}


\section{Proof of Theorem~\ref{thm:intro1}}
\label{sec:bumaginwise}
In this section we apply Theorem~\ref{theorem:secondmainthm} to prove the main results of this paper, Theorems~\ref{thm:intro1}~and~\ref{thm:intro2}.

\p{Sapir's embedding} To apply Theorem~\ref{theorem:secondmainthm} we need to have some knowledge or control over the centralisers of elements in $\out(H)$. To do this, we use an embedding of Sapir \cite[Theorem~5.1]{sapir2014higman}. If $K$ is a finitely generated, recursively presented group and $x\in K$, then Sapir's embedding gives a finitely presented group $P$ containing $K$ as a subgroup and such that $C_K(x)=C_P(x)$. It is an open problem of Osin that every recursively presented group can be embedded as a malnormal subgroup of a finitely presented group \cite{sapir2014higman}. Sapir remarks that in his embedding $K$ is malnormal in $P$, hence the open problem of Osin has a positive solution, and that this will be proven in his next paper. The proofs of Theorems~\ref{thm:intro1}~and~\ref{thm:intro2} both apply Sapir's embedding. Note that Theorem~\ref{thm:intro2} can be rephrased as ``if Osin's problem admits a positive solution then every finitely generated, recursively presentable group can be realised as the outer automorphism group of a finitely generated, residually finite group''.

\p{The Bumagin--Wise question}
We now prove two theorems, which combine to prove Theorem~\ref{thm:intro1} and the second of which yields Theorem~\ref{thm:intro2}. The first theorem, Theorem~\ref{theorem:BWanswer1}, gives a partial answer to Bumagin--Wise's question for certain groups, while the second theorem, Theorem~\ref{theorem:BWanswer2}, gives a complete answer to Bumagin--Wise's question for certain groups.

The proofs of Theorems~\ref{theorem:BWanswer1}~and~\ref{theorem:BWanswer2} both use the fact that a split extension of a finitely generated, residually finite group by a residually finite group is residually finite \cite{mal1956homomorphisms}. Hence if the base group $H$ is a finitely generated, residually finite group then the mapping torus $H_{\phi}=H\rtimes_{\phi}\mathbb{Z}$ is also a finitely generated, residually finite group.

Recall that the group $\out^0(H_{\phi})$, as defined in Section~\ref{sec:SubgroupWeIsolate}, is the subgroup of $\out(H_{\phi})$ consisting of the elements of the form $\widehat{\alpha}_{\delta}$, where ${\alpha}_{\delta}$ is defined in Lemma~\ref{lem:alphaform}.

\begin{theorem}
\label{theorem:BWanswer1}
Let $Q$ be a finitely generated, recursively presented group. Then there exists a finitely generated, residually finite group $G$ such that $\out(G)\cong Q\times C_2$.
\end{theorem}

\begin{proof}
Define $Q_2=Q\times C_2$. As $Q$ is finitely generated and recursively presented, we can use Sapir's embedding to construct a finitely presented group $P$ which contains $Q_2$ and such that $C_P(k)=Q_2$ where $k$ is the generator of the $C_2$ factor of $Q_2$. As $P$ is finitely presented, there exists a finitely generated, residually finite group $H$ such that $\out(H)\cong P$ \cite{BumaginWise2005}. Note that this group $H$ is generated by elements of finite order, and so does not map onto $\mathbb{Z}$, and also note that $H$ is a non-cyclic subgroup of a finitely presented $C^{\prime}(1/6)$ group and therefore has trivial center \cite{BumaginWise2005}. Thus, Theorem~\ref{theorem:secondmainthm} is applicable to $H_{\phi}=H\rtimes_{\phi}\mathbb{Z}$ for all $\phi\in\aut(H)$.

Let $\widehat{\phi}$ be the element of $\out(H)$ associated to $k\in Q_2$. Thus, $C_{\out(H)}(\widehat{\phi})\cong Q_2$. Form $H_{\phi}=H\rtimes_{\phi}\mathbb{Z}$ for some $\phi\in\widehat{\phi}$. Then $\out^0(H_{\phi})\cong Q$ by Theorem~\ref{theorem:secondmainthm}. Note that $H_{\phi}$ is finitely generated, and residually finite \cite{mal1956homomorphisms}.

To complete the theorem it is sufficient to prove that $\out(G)=\out^0(H_{\phi})\times C_2$. To see this, note that $k=k^{-1}$. Thus, the automorphism $\psi: h\mapsto h, t\mapsto t^{-1}$ can be taken as the coset representative for $\out(H_{\phi})/\out^0(H_{\phi})$. This automorphism has order two and generates a normal subgroup of $\out(H_{\phi})$. Therefore, taking $G=H_{\phi}$, $\out(G)=\out^0(G)\times\langle\widehat{\psi}\rangle\cong Q\times C_2$, as required.
\end{proof}

The following theorem, Theorem~\ref{theorem:BWanswer2}, allows us to apply a positive solution of Osin's problem to get a positive solution to Bumagin--Wise's question for finitely generated, recursively presented groups. This is because if $Q$ is finitely generated and recursively presented then the conditions of Theorem~\ref{theorem:BWanswer2} hold if, for example, $Q\times C_3$ embeds malnormally into a finitely presented group, and a positive solution to Osin's question gives us this embedding.

\begin{theorem}
\label{theorem:BWanswer2}
Let $Q^{\prime}=Q\times C$ where $C=\<k\>$ is cyclic of order greater than two (possibly infinite). Suppose that $Q^{\prime}$ can be embedded into a
finitely presented group $P$ where $k$ is not conjugate to $k^{-1}$ in $P$. Then there exists a finitely generated, residually finite group $G$ such that $\out(G)\cong Q$.
\end{theorem}

\begin{proof}
Write $H$ for the finitely generated, residually finite group such that $\out(H)\cong P$ \cite{BumaginWise2005}, and, as in the proof of Theorem~\ref{theorem:BWanswer1}, take $\widehat{\phi}$ to be the element of $\out(H)$ associated to $k\in Q^{\prime}$ and form the finitely generated, residually finite group $G\cong H\rtimes_{\phi}\mathbb{Z}$ such that $\out^0(G)\cong Q^{\prime}$. Finally, because $k$ is not conjugate to $k^{-1}$ in $P$, Lemma~\ref{lem:index1or2} allows us to conclude that $\out(G)=\out^0(G)\cong Q$, as required.
\end{proof}

We shall now prove Theorems~\ref{thm:intro1}~and~\ref{thm:intro2}.
 
\begin{proof}[Proof of Theorem~\ref{thm:intro1}]
Theorem~\ref{thm:intro1} follows immediately from combining Theorems~\ref{theorem:BWanswer1}~and~\ref{theorem:BWanswer2}.
\end{proof}
 
Note that if Osin's open problem has a positive solution, so every finitely generated, recursively presented group is a malnormal subgroup of a finitely presented group, then we can use Theorem~\ref{theorem:BWanswer2} and disregard Theorem~\ref{theorem:BWanswer1} to obtain Theorem~\ref{thm:intro2}.

\begin{proof}[Proof of Theorem~\ref{thm:intro2}]
If $Q$ is a finitely generated, recursively presentable group then, by the assumptions of the theorem, $Q\times C_3$ embeds malnormally into a finitely presented group $P$. Theorem~\ref{thm:intro2} then follows from Theorem~\ref{theorem:BWanswer2}.
\end{proof}

\p{Recursive presentability}
It is natural to ask how far Theorem~\ref{thm:intro2} goes towards solving Bumagin--Wise's question, assuming that Osin's problem has a positive solution and that the groups $Q$ in the statement of the question are additionally assumed to be finitely generated. The ``best possible'' case would be that every finitely generated group which occurrs as the outer automorphism group of a finitely generated, residually finite group is recursively presentable, and so Theorem~\ref{thm:intro2} would be the complete solution to Bumagin--Wise's question for finitely generated groups. However, the following proposition, Proposition~\ref{prop:fgrf}, implies that this case does not happen. We then prove, in Proposition~\ref{corol:RESTRICTEDfullsolution}, that if the groups $G_Q$ in the statement of Bumagin--Wise's question are additionally assumed to be recursively presentable then Theorem~\ref{thm:intro2} is the complete solution to Bumagin--Wise's question for finitely generated groups.
\begin{proposition}
\label{prop:fgrf}
There exists a finitely generated, non-recursively presentable group $Q$ which can be realised as the outer automorphism group of a finitely generated, residually finite group $G_Q$.
\end{proposition}

We now explain the proof of Proposition~\ref{prop:fgrf}. Note that there exists a finitely generated, residually finite group $R$ which is not recursively presentable (Bridson--Wilton \cite{bridson2013triviality} point out that this follows from work of Slobodsko\v\i \cite{Slobodskoi1981undecidability}). Using the existence of such a non-recursively presented group $R$, a forthcoming paper of the author (see also the author's PhD thesis \cite[Corollary 4.3.16]{logan2014outer}) constructs a finitely generated, residually finite group $G_{\widehat{R}}$ whose outer automorphism group is finitely generated but not recursively presentable (indeed, $R$ is embedded with finite index into $\out(G_{\widehat{R}})$). This proves Proposition~\ref{prop:fgrf}. Note, however, this group $G_{\widehat{R}}$ is itself not recursively presentable.

We now provide a positive answer to the following question: assuming Osin's problem has a positive solution, is it true that a finitely generated group $Q$ can be realised as the outer automorphism group of a recursively presented, finitely generated, residually finite group $G_Q$ if and only if $Q$ is recursively presentable? That is, is Theorem~\ref{thm:intro2} is the complete solution to this restricted version of Bumagin--Wise's question? We provide a positive answer by combining Theorem~\ref{thm:intro2} with following proposition, Proposition~\ref{prop:Minasyan}, which is due to Ashot Minasyan in a private communication with the author. Proposition~\ref{prop:Minasyan} also explains why the group $G_{\widehat{R}}$ in the author's construction, cited above, is not recursively presentable. We state the proposition, give a sketch proof of it, and then combine it with Theorem~\ref{thm:intro2} to prove Proposition~\ref{corol:RESTRICTEDfullsolution}.

\begin{proposition}[A. Minasyan]
\label{prop:Minasyan}
Suppose that $G$ is a finitely generated, recursively presentable group whose outer automorphism group $\out(G)$ is also finitely generated. Then $\aut(G)$ and $\out(G)$ are recursively presentable.
\end{proposition}

\begin{proof}
Assume that $G$ and $\out(G)$ are finitely generated, and that $G$ is recursively presented. Note that this implies that $\aut(G)$ is finitely generated, as it is an extension of $\inn(G)$ by $\out(G)$, and both $\inn(G)\cong G/Z(G)$ and $\out(G)$ are assumed to be finitely generated. We shall just prove that $\aut(G)$ is recursively presentable; this implies that $\out(G)$ is recursively presentable because $G$ and $\out(G)$ are both assumed to be finitely generated.

To prove that $\aut(G)$ is recursively presentable, we shall start with the generators of $\aut(G)$ and construct an algorithm which lists all the relators of $\aut(G)$. Let $\psi_1, \ldots, \psi_m$ be a generating set for $\aut(G)$ and let $x_1, \ldots, x_n$ be a generating set for $G$. We can assume that for each $i,j$ we know words $u_{ij}$ and $v_{ij}$, over the alphabet $\{x_1,\ldots,x_n\}^{\pm 1}$, such that $\psi_i(x_j)=u_{ij}$ and $\psi_i^{-1}(x_j)=v_{ij}$, because this is a finite collection of words and we are only proving the existence of an algorithm. Now, since $G$ is recursively presented, there is a (partial) algorithm $A$ which takes on input a pair of words, $(w_1, w_2)$  say, over $\{x_1,\ldots,x_n\}^{\pm 1}$ and stops, outputting ``yes'' if and only if $w_1=w_2$ in $G$ (this is the algorithm $A$ which re-writes $w_1$ in all possible ways and compares the result with $w_2$).

To obtain an algorithm listing all defining relators of $\aut(G)$, start enumerating all words $\Psi_1,\Psi_2,\ldots$ over $\{\psi_1,\ldots,\psi_m\}^{\pm 1}$. At the same time, for every $k$ check if $\Psi_k(x_j)=x_j$ in $G$ for all $j=1,\ldots,n$ (by writing $\Psi_k(x_j)$ as a word $W$ over the generators $\{x_1,\ldots,x_n\}^{\pm 1}$ in terms of $u_{ij}$ and $v_{ij}$, and then inputting the pair $(W, x_j)$ into $A$). If $\Psi_k=1$ in $\aut(G)$ then we will verify this in finite time, and so we can add $\Psi_k$ to the list of defining relators of $\aut(G)$. Thus we have an algorithm listing all relators in $\aut(G)$, and conclude that $\aut(G)$ is recursively presented.
\end{proof}

Proposition~\ref{prop:Minasyan} and Theorem~\ref{thm:intro2} can be combined to yield the following result. Note that combining Proposition~\ref{prop:Minasyan} and Theorem~\ref{thm:intro1} yields a similar result which is independent of Osin's problem.

\begin{proposition}
\label{corol:RESTRICTEDfullsolution}
Suppose every finitely generated, recursively presented group $Q$ can be embedded as a malnormal subgroup of a finitely presented group $H_Q$. Then a finitely generated group $Q$ can be realised as the outer automorphism group of a recursively presented, finitely generated, residually finite group $G_Q$ if and only if $Q$ is recursively presentable.
\end{proposition}


\bibliographystyle{amsalpha}
\bibliography{BibTexBibliography}

\providecommand{\bysame}{\leavevmode\hbox to3em{\hrulefill}\thinspace}
\providecommand{\MR}{\relax\ifhmode\unskip\space\fi MR }
\providecommand{\MRhref}[2]{%
  \href{http://www.ams.org/mathscinet-getitem?mr=#1}{#2}
}
\providecommand{\href}[2]{#2}
\begin{thebibliography}{DGG01}

\bibitem[ALM14]{arzhantseva2014isomorphism}
G.~Arzhantseva, J.~Lafont, and A.~Minasyan, \emph{Isomorphism versus
  commensurability for a class of finitely presented groups}, J. Group Theory
  \textbf{17} (2014), no.~2, 361--378.

\bibitem[BG03]{braun2003outer}
G.~Braun and R.~G{\"o}bel, \emph{Outer automorphisms of locally finite
  p-groups}, J. Algebra \textbf{264} (2003), no.~1, 55--67.

\bibitem[BW05]{BumaginWise2005}
I.~Bumagin and D.T. Wise, \emph{Every group is an outer automorphism group of a
  finitely generated group}, J. Pure App. Alg. \textbf{200} (2005), no.~1,
  137--147.

\bibitem[BW14]{bridson2013triviality}
M.R. Bridson and H.~Wilton, \emph{The triviality problem for profinite
  completions}, Preprint, arXiv:1401.2790 (2014).

\bibitem[DGG01]{droste2001all}
M.~Droste, M.~Giraudet, and R.~G{\"o}bel, \emph{All groups are outer
  automorphism groups of simple groups}, J. Lon. Math. Soc. \textbf{64} (2001),
  no.~3, 565--575.

\bibitem[FM05]{frigerio2005countable}
R.~Frigerio and B.~Martelli, \emph{Countable groups are mapping class groups of
  hyperbolic 3-manifolds}, Math. Res. Lett. 13 (2006), 897-910 (2005).

\bibitem[GP00]{gobel2000outer}
R.~G{\"o}bel and A.T. Paras, \emph{Outer automorphism groups of metabelian
  groups}, J. Pure App. Alg. \textbf{149} (2000), no.~3, 251--266.

\bibitem[Koj88]{kojima1988isometry}
S.~Kojima, \emph{Isometry transformations of hyperbolic 3-manifolds}, Topol.
  Appl. \textbf{29} (1988), no.~3, 297--307.

\bibitem[Log14]{logan2014outer}
A.~D. Logan, \emph{The outer automorphism groups of three classes of groups},
  Ph.D. thesis, University of Glasgow, 2014.

\bibitem[Mal56]{mal1956homomorphisms}
A.I. Mal'cev, \emph{On homomorphisms onto finite groups}, Uchen. Zap. Ivanov
  Gos. Ped. Inst. \textbf{18} (1956), 49--60, [English transl. in: Trans. Amer.
  Math. Soc. Transl. {\bf 2}, 119 (1983) 67--79].

\bibitem[Mat89]{matumoto1989any}
T.~Matumoto, \emph{Any group is represented by an outer automorphism group},
  Hiroshima Math. J. \textbf{19} (1989), no.~1, 209--219.

\bibitem[Min09]{minasyan2009groups}
A.~Minasyan, \emph{Groups with finitely many conjugacy classes and their
  automorphisms}, Comm. Math. Helv. \textbf{84} (2009), no.~2, 259--296.

\bibitem[Rip82]{rips1982subgroups}
E.~Rips, \emph{Subgroups of small cancellation groups}, Bull. London Math. Soc
  \textbf{14} (1982), no.~1, 45--47.

\bibitem[Sap14]{sapir2014higman}
Mark Sapir, \emph{A {Higman} embedding preserving asphericity}, J. Am. Math.
  Soc. \textbf{27} (2014), no.~1, 1--42.

\bibitem[Slo81]{Slobodskoi1981undecidability}
A.~M. Slobodsko\v\i, \emph{Undecidability of the universal theory of finite
  groups}, Algebra i Logika \textbf{20} (1981), no.~2, 207--230, 251.

\end{thebibliography}
\end{document}